\documentclass[12pt,reqno]{amsart}

\newtheorem{theorem}{Theorem}[section]
\newtheorem{proposition}[theorem]{Proposition}
\newtheorem{lemma}[theorem]{Lemma}
\newtheorem{corollary}[theorem]{Corollary}

\numberwithin{equation}{section}

\begin{document}

\baselineskip=16pt

\title[Abel-Jacobi map and curvature of the pulled back metric]{Abel-Jacobi map and
curvature of the pulled back metric}

\author[I. Biswas]{Indranil Biswas}

\address{School of Mathematics, Tata Institute of Fundamental
Research, Homi Bhabha Road, Mumbai 400005, India}

\email{indranil@math.tifr.res.in}



\date{}

\begin{abstract}
Let $X$ be a compact connected Riemann surface of genus at least two. The Abel-Jacobi
map $\varphi\, :\, {\rm Sym}^d(X)\, \longrightarrow\, {\rm Pic}^d(X)$ is an embedding if $d$ is less
than the gonality of $X$. We investigate the curvature of the pull-back, by $\varphi$, of the flat metric
on ${\rm Pic}^d(X)$. In particular, we show that when $d\,=\,1$, the curvature is strictly negative
everywhere if $X$ is not hyperelliptic, and when $X$ is hyperelliptic, the curvature is nonpositive
with vanishing exactly on the points of $X$ fixed by the hyperelliptic involution.
\end{abstract}

\maketitle

\section{Introduction}\label{intro.}

Let $X$ be a compact connected Riemann surface of genus $g$, with $g\, \geq\, 2$. The gonality of $X$ is defined to be
the smallest integer $\gamma_X$ such that there is a nonconstant holomorphic map from $X$ to
${\mathbb C}{\mathbb P}^1$ of degree $\gamma_X$. Consider the Abel-Jacobi
map $$\varphi\, :\, {\rm Sym}^d(X)\, \longrightarrow\, {\rm Pic}^d(X)$$ that sends an effective divisor $D$ on $X$ of degree $d$
to the corresponding
holomorphic line bundle ${\mathcal O}_X(D)$. If $d\, <\, \gamma_X$, then $\varphi$ is an embedding (Lemma
\ref{lem1}). On the other hand, ${\rm Pic}^d(X)$ is equipped with a flat K\"ahler form, which we will denote by $\omega_0$.
So, $\varphi^*\omega_0$ is a K\"ahler form on ${\rm Sym}^d(X)$, whenever $d\, <\, \gamma_X$.
The K\"ahler metric $\varphi^*\omega_0$ on ${\rm Sym}^d(X)$ is relevant in the study of abelian
vortices (see \cite{Ri}, \cite{BR} and references therein).

Our aim here is to study the curvature of this K\"ahler form $\varphi^*\omega_0$ on ${\rm Sym}^d(X)$.

Consider the $g$--dimensional vector space $H^0(X,\, K_X)$ consisting
of holomorphic one-forms on $X$. It is equipped with a natural
Hermitian structure. Let
$$
{\mathbb G}\,=\, {\rm Gr}(d, H^0(X,K_X))
$$
be the Grassmannian parametrizing all $d$ dimensional quotients of $H^0(X,\, K_X)$. The Hermitian structure on
$H^0(X,\, K_X)$ produces a Hermitian structure on the tautological
vector bundle on ${\mathbb G}$ of rank $d$; this
tautological bundle on ${\mathbb G}$ of rank $d$ will be denoted by $V$. The Hermitian structure on
$H^0(X,\, K_X)$ also gives a Fubini--Study K\"ahler form on ${\mathbb G}$.

There is a natural holomorphic map
$$
\rho\, :\, \text{Sym}^d(X)\,\longrightarrow\, {\mathbb G}
$$
(see \eqref{e8}). We prove that the holomorphic Hermitian vector bundle $\rho^*V\, \longrightarrow\,
{\rm Sym}^d(X)$ is isomorphic to the holomorphic cotangent bundle $\Omega_{{\rm Sym}^d(X)}$ equipped with the
Hermitian structure given by $\varphi^*\omega_0$ (Theorem \ref{thm1}).

Since the curvature of the holomorphic Hermitian vector bundle $V\, \longrightarrow\, {\mathbb 
G}$ is standard, Theorem \ref{thm1} gives a description of the curvature of $\varphi^*\omega_0$ 
in terms of $\rho$. In particular, we show that when $d\,=\, 1$, the curvature of 
$\varphi^*\omega_0$ is strictly negative if $X$ is not hyperelliptic; on the other hand, if $X$ 
is hyperelliptic, then the curvature of $\varphi^*\omega_0$ vanishes at the $2(g+1)$ points of 
$X$ fixed by the hyperelliptic involution; the curvature of $\varphi^*\omega_0$ is strictly 
negative outside these $2(g+1)$ points (Proposition \ref{prop1}).

\section{Gonality and flat metric}

As before, $X$ is a compact connected Riemann surface of genus $g$, with $g\, \geq\, 2$. For any 
positive integer $d$, let $\text{Sym}^d(X)$ denote the quotient of the Cartesian product $X^d$ 
under the natural action of the group of permutations of $\{1,\, \cdots, \, n\}$. This 
$\text{Sym}^d(X)$ is a smooth complex projective manifold of dimension $d$. The component of the 
Picard group of $X$ parametrizing the holomorphic line bundles of degree $d$ will be denoted by 
$\text{Pic}^d(X)$.

Let $X^d\, \longrightarrow\, \text{Pic}^d(X)$ be the map that sends any $(x_1,\, \cdots,\, x_d)
\,\in\, X^d$ to the line bundle ${\mathcal O}_X(x_1+\ldots +x_d)$. It descends to a map
\begin{equation}\label{e1}
\varphi\, :\, \text{Sym}^d(X)\, \longrightarrow\, \text{Pic}^d(X)\, .
\end{equation}
This map $\varphi$ is surjective if and only if $d\, \geq\, g$.

Consider the space of all nonconstant holomorphic maps from $X$ to the complex projective line
${\mathbb C}{\mathbb P}^1$. More precisely, consider the degree of all such maps. Since $g\, \geq\, 2$, the
degree of any such map is at least two.
The {\it gonality} of $X$ is defined to the smallest integer among the degrees of maps in this space
\cite[p.~171]{Ei}. Equivalently, the gonality of $X$ is the smallest one among the degrees of holomorphic
line bundles $L$ on $X$ with $\dim H^0(X,\, L)\, \geq\, 2$. The gonality of $X$ will be denoted by
$\gamma_X$. Note that $\gamma_X\,=\, 2$ if and only if $X$ is hyperelliptic. 
The gonality of a generic compact Riemann surface of genus $g$ is $\left \lfloor{\frac{g+3}{2}}\right \rfloor$.

\begin{lemma}\label{lem1}
Assume that $d\, < \, \gamma_X$. Then the map $\varphi$ in \eqref{e1} is an embedding.
\end{lemma}

\begin{proof}
We will first show that $\varphi$ is injective.

Take any point ${\mathbf x} \, :=\,\{x_1,\, \cdots\, , x_d\}\, \in\, \text{Sym}^d(X)$; the points
$x_i$ need not be distinct. The divisor $\sum_{i=1}^d x_i$ will be denoted by $D_{\mathbf x}$.
If ${\mathbf y} \, :=\,\{y_1,\, \cdots\, , y_d\}\, \in\, \text{Sym}^d(X)$ is another point
such that the line bundles ${\mathcal O}_X(D_{\mathbf x})$ and ${\mathcal O}_X(D_{\mathbf y})$ are isomorphic,
where $D_{\mathbf y}\,=\, \sum_{i=1}^d y_i$, then there is a meromorphic function on $X$ with
pole divisor $D_{\mathbf y}$ and zero divisor $D_{\mathbf x}$. In particular, the degree of this
meromorphic function is $d$. But this contradicts the given condition
that $d\, < \, \gamma_X$. Consequently, the map $\varphi$ is injective.

We need to show that for ${\mathbf x}\, \in\, \text{Sym}^d(X)$, the differential
\begin{equation}\label{e2}
d\varphi({\mathbf x})\, :\, T_{\mathbf x}\text{Sym}^d(X)\, \longrightarrow\, T_{\varphi({\mathbf x})}\text{Pic}^d(X)
\,=\, H^1(X,\, {\mathcal O}_X)
\end{equation}
is injective.

We will quickly recall a description of the tangent bundle $T\text{Sym}^d(X)$.

Let
$$
{\mathcal D}\, \subset\, \text{Sym}^d(X)\times X
$$
be the tautological reduced effective divisor consisting of all $(\{y_1,\, \cdots\, , y_d\},\, y)$ such that
$y\, \in\, \{y_1,\, \cdots\, , y_d\}$. The projection of $\text{Sym}^d(X)\times X$ to $\text{Sym}^d(X)$
will be denoted by $p$. Consider the quotient sheaf
$$
{\mathcal O}_{\text{Sym}^d(X)\times X}({\mathcal D})/{\mathcal O}_{\text{Sym}^d(X)\times X}\, \longrightarrow\,
\text{Sym}^d(X)\times X\, .
$$
Note that its support is the divisor
$\mathcal D$. The tangent bundle $T\text{Sym}^d(X)$ is the direct image
$$
p_*({\mathcal O}_{\text{Sym}^d(X)\times X}({\mathcal D})/{\mathcal O}_{\text{Sym}^d(X)\times X})\, \longrightarrow\,
\text{Sym}^d(X)\, .
$$

Take any ${\mathbf x} \, :=\,\{x_1,\, \cdots\, , x_d\}\, \in\, \text{Sym}^d(X)$. Let
\begin{equation}\label{e3}
0\, \longrightarrow\, {\mathcal O}_X(-D_{\mathbf x})\, \longrightarrow\, {\mathcal O}_X
\, \longrightarrow\, \widetilde{Q}({\mathbf x})\,:=\, {\mathcal O}_X/{\mathcal O}_X(-D_{\mathbf x})
\, \longrightarrow\, 0
\end{equation}
be the short exact sequence of sheaves on $X$, where $D_{\mathbf x}$, as before, is the effective divisor
given by ${\mathbf x}$. Tensoring the sequence in \eqref{e3}
with the line bundle ${\mathcal O}_X(-D_{\mathbf x})^*\,=\, {\mathcal O}_X(D_{\mathbf x})$ we obtain the
following short exact sequence of sheaves on $X$:
\begin{equation}\label{c1}
0\, \longrightarrow\, End({\mathcal O}_X(-D_{\mathbf x}))\,=\, {\mathcal O}_X \, \longrightarrow\,
Hom({\mathcal O}_X(-D_{\mathbf x}),\, {\mathcal O}_X)
\end{equation}
$$
=\, {\mathcal O}_X(D_{\mathbf x})\, \longrightarrow\, Q({\mathbf x})\,:=\,
Hom({\mathcal O}_X(-D_{\mathbf x})\, ,\widetilde{Q}({\mathbf x}))\, \longrightarrow\, 0\, .
$$
Let
\begin{equation}\label{e4}
0\, \longrightarrow\, H^0(X,\, {\mathcal O}_X)\,\stackrel{\alpha}{\longrightarrow}\,
H^0(X,\, {\mathcal O}_X(D_{\mathbf x}))\,\stackrel{\beta}{\longrightarrow}\, H^0(X,\,
Q(\mathbf{x}))\,\stackrel{\delta_{\mathbf x}}{\longrightarrow}\, H^1(X,\, {\mathcal O}_X)
\end{equation}
$$
\stackrel{\nu}{\longrightarrow}\, H^1(X,\, {\mathcal O}_X(D_{\mathbf x}))
\, \longrightarrow\, H^1(X,\, Q(\mathbf{x}))\,=\,0
$$
be the long exact sequence of cohomologies associated to the short exact sequence
of sheaves in \eqref{e4}. From the earlier description of $T\text{Sym}^d(X)$ we have
the following:
\begin{equation}\label{y1}
T_{\mathbf x}\text{Sym}^d(X)\,=\, H^0(X,\, Q({\mathbf x}))\, .
\end{equation}

Since $d\, < \, \gamma_X$, we have
\begin{equation}\label{z1}
H^0(X,\, {\mathcal O}_X(D_{\mathbf x}))\,=\, \mathbb C\, .
\end{equation}
Hence the homomorphism
$\alpha$ in \eqref{e4} is an isomorphism. This implies that the homomorphism
$\delta_{\mathbf x}$ in \eqref{e4} is injective. So
the exact sequence in \eqref{e4} gives the exact sequence
\begin{equation}\label{e5}
0 \,\longrightarrow\, T_{\mathbf x}\text{Sym}^d(X)\,=\, H^0(X,\,
Q(\mathbf{x}))\,\stackrel{\delta_{\mathbf x}}{\longrightarrow}\, H^1(X,\, {\mathcal O}_X)\, .
\end{equation}

The tangent bundle of $\text{Pic}^d(X)$ is the trivial vector bundle over $\text{Pic}^d(X)$
with fiber $H^1(X,\, {\mathcal O}_X)$. The differential
$d\varphi({\mathbf x})$ in \eqref{e2} coincides with the homomorphism $\delta_{\mathbf x}$
in \eqref{e4}. Since $\delta_{\mathbf x}$ in \eqref{e5} is injective, it
follows that $d\varphi({\mathbf x})$ is injective.
\end{proof}

Let $K_X$ denote the holomorphic cotangent bundle of $X$. The vector space $H^0(X,\, K_X)$
is equipped with the Hermitian form
\begin{equation}\label{z2}
\langle\theta_1,\, \theta_2\rangle \,:=\, \int_X \theta_1\wedge \overline{\theta_2}\,\in\, {\mathbb C}\, , \ \
\theta_1,\, \theta_2\, \in\, H^0(X,\, K_X)\, .
\end{equation}
This Hermitian form on $H^0(X,\, K_X)$ produces a Hermitian form on the dual vector space
$$H^0(X,\, K_X)^*\,=\, H^1(X,\, {\mathcal O}_X)\, ;$$ this isomorphism is given by Serre duality.
This Hermitian form on $H^1(X,\, {\mathcal O}_X)$ produces a K\"ahler
structure on $\text{Pic}^d(X)$ which is invariant under the translation action of
$\text{Pic}^0(X)$ on $\text{Pic}^d(X)$. This K\"ahler structure on $\text{Pic}^d(X)$ will
be denoted by $\omega_0$.

Now Lemma \ref{lem1} has the following corollary:

\begin{corollary}\label{cor1}
Assume that $d\, < \, \gamma_X$. Then $\varphi^*\omega_0$ is a K\"ahler structure on 
${\rm Sym}^d(X)$.
\end{corollary}

\section{Mapping to a Grassmannian}

We will always assume that $d\, < \, \gamma_X$. Since $\gamma_X\, \leq\, g$, we have
$d\, <\, g$.

Let
\begin{equation}\label{e6}
{\mathbb G}\,=\, {\rm Gr}(d, H^0(X,K_X))
\end{equation}
be the Grassmannian parametrizing all $d$ dimensional quotients of $H^0(X,\, K_X)$. Let
\begin{equation}\label{z3}
V\, \longrightarrow\, {\mathbb G}
\end{equation}
be the tautological vector bundle of rank $d$. So $V$ is a quotient of the trivial vector
bundle ${\mathbb G}\times H^0(X,\, K_X)\, \longrightarrow\, {\mathbb G}$. Consider the
Hermitian form on $H^0(X,\, K_X)$ defined in \eqref{z2}. It produces a Hermitian structure on
the trivial vector bundle ${\mathbb G}\times H^0(X,\, K_X)\, \longrightarrow\, {\mathbb G}$.
Identifying the quotient $V$ with the orthogonal complement of the kernel
of the projection to $V$, we get a
Hermitian structure on $V$. Let
\begin{equation}\label{z4}
H_0\, :\, V\otimes \overline{V}\, \longrightarrow\, {\mathbb G}\times \mathbb C
\end{equation}\
be this Hermitian structure on $V$.

Take any ${\mathbf x} \, :=\,\{x_1,\, \cdots\, , x_d\}\, \in\, \text{Sym}^d(X)$.
Consider the short exact sequence of sheaves on $X$
\begin{equation}\label{c2}
0\, \longrightarrow\, K_X\otimes {\mathcal O}_X(-D_{\mathbf x})\, \longrightarrow\, K_X
\, \longrightarrow\, Q'({\mathbf x})\,:=\, K_X/(K_X\otimes {\mathcal O}_X(-D_{\mathbf x}))
\, \longrightarrow\, 0\, ,
\end{equation}
where $D_{\mathbf x}$ as before is the divisor on $X$ given by $\mathbf x$. Let
\begin{equation}\label{e7}
0\, \longrightarrow\, H^0(X,\, K_X\otimes {\mathcal O}_X(-D_{\mathbf x}))
\,\stackrel{\nu'}{\longrightarrow}\, H^0(X,\, K_X)\,\stackrel{\delta'_{\mathbf x}}{\longrightarrow}\, H^0(X,\, Q'({\mathbf x}))
\end{equation}
$$
\,\stackrel{\beta'}{\longrightarrow}\, H^1(X,\, K_X\otimes {\mathcal O}_X(-D_{\mathbf x}))
\, \stackrel{\alpha'}{\longrightarrow}\, H^1(X,\, K_X) \,\longrightarrow\, H^1(X,\, Q'({\mathbf x}))\,=\, 0
$$
be the long exact sequence of cohomologies associated to it. By Serre duality,
$$
H^1(X,\, K_X\otimes {\mathcal O}_X(-D_{\mathbf x}))\,=\, H^0(X,\, {\mathcal O}_X(D_{\mathbf x}))^*\, ;
$$
hence from \eqref{z1} it follows that $\alpha'$ in \eqref{e7} is an isomorphism. This implies that
$\beta'$ in \eqref{e7} is the zero homomorphism, hence $\delta'_{\mathbf x}$ is surjective. In other words,
$H^0(X,\, Q'({\mathbf x}))$ is a quotient of
$H^0(X,\, K_X)$ of dimension $d$. Therefore, $H^0(X,\, Q'({\mathbf x}))$ gives a point of
the Grassmannian ${\mathbb G}$ constructed in \eqref{e6}.

Let
\begin{equation}\label{e8}
\rho\, :\, \text{Sym}^d(X)\,\longrightarrow\, {\mathbb G}
\end{equation}
be the morphism defined by ${\mathbf x} \,\longmapsto\, H^0(X,\, Q'({\mathbf x}))$.

\begin{theorem}\label{thm1}\mbox{}
\begin{enumerate}
\item The vector bundle $\rho^*V\,\longrightarrow\,{\rm Sym}^d(X)$, where $V$ and $\rho$ are constructed in
\eqref{z3} and \eqref{e8} respectively, is holomorphically identified with the holomorphic cotangent bundle
$\Omega_{{\rm Sym}^d(X)}$.

\item Using the identification in (1), the Hermitian structure $\rho^*H_0$, where $H_0$ is constructed in \eqref{z4}, coincides
with the Hermitian structure on $\Omega_{{\rm Sym}^d(X)}$ given by $\varphi^*\omega_0$ in Corollary \ref{cor1}.
\end{enumerate}
\end{theorem}

\begin{proof}
We will show that the homomorphisms $\alpha'$, $\beta'$, $\delta'_{\mathbf x}$ and $\nu'$ in \eqref{e7} are duals of the
homomorphisms $\alpha$, $\beta$, $\delta_{\mathbf x}$ and $\nu$ respectively, which are constructed in \eqref{e4}. This
actually follows from the fact that the complex in \eqref{c2} is dual of the complex in \eqref{c1}. We will
elaborate this a bit.

By Serre duality, we have
\begin{equation}\label{di}
H^1(X,\, {\mathcal O}_X)^*\,=\, H^0(X,\, K_X)\ \ \text{ and }\ \ 
H^1(X,\, {\mathcal O}_X(D_{\mathbf x}))^*\,=\, H^0(X,\, K_X\otimes {\mathcal O}_X(-D_{\mathbf x}))\, .
\end{equation}
Using these isomorphisms, the homomorphism $\nu'$ in \eqref{e7} is the dual of the homomorphism $\nu$ in \eqref{e4}.
Therefore, from \eqref{e7} and \eqref{e4} we have
\begin{equation}\label{e9}
H^0(X,\, Q(\mathbf{x}))^*\,=\, H^0(X,\, Q'({\mathbf x}))\, .
\end{equation}
But $H^0(X,\, Q(\mathbf{x}))\,=\, T_{\mathbf x}\text{Sym}^d(X)$ (see \eqref{y1}). On the other hand,
$H^0(X,\, Q'({\mathbf x}))$ is the fiber of $V$ over the point $\rho({\mathbf x})\, \in\, {\mathbb G}$.
Therefore, the first statement of the theorem follows from \eqref{e9}.

The isomorphism $H^1(X,\, {\mathcal O}_X)^*\,=\, H^0(X,\, K_X)$
in \eqref{di} is an isometry, because the Hermitian
form on $H^1(X,\, {\mathcal O}_X)^*$ is defined using the Hermitian form on $H^0(X,\, K_X)$ in \eqref{z2} and
this isomorphism. This implies that the isomorphism in \eqref{e9} is an isometry, after
$H^0(X,\, Q(\mathbf{x}))$ (respectively, $H^0(X,\, Q'({\mathbf x}))$) is equipped with the
Hermitian structure obtained from the Hermitian structure on $H^1(X,\, {\mathcal O}_X)$
(respectively, $H^0(X,\, K_X)$) using \eqref{e4} (respectively, \eqref{e7}). This completes the proof.
\end{proof}

Take $d\, =\,1$. Consider the K\"ahler form $\varphi^*\omega_0$ on $X$ in Corollary \ref{cor1}. Let
$\Theta\cdot \varphi^*\omega_0$ be the curvature of $\varphi^*\omega_0$; so $\Theta$ is a smooth
function on $X$.

\begin{proposition}\label{prop1}\mbox{}
The curvature function $\Theta$ is nonpositive.
\begin{enumerate}
\item If $X$ is not hyperelliptic, then $\Theta$ is strictly negative everywhere on $X$.

\item If $X$ is hyperelliptic, then $\Theta$ is strictly negative everywhere outside the $2(g+1)$
points fixed by hyperelliptic involution of $X$. The function $\Theta$ vanishes on the $2(g+1)$
fixed points of the hyperelliptic involution.
\end{enumerate}
\end{proposition}

\begin{proof}
Take a complex vector space $W$ equipped with a Hermitian form. Let ${\mathbb P}(W)$ be
the projective space that parametrizes quotients of $W$ of dimension one. The curvature of the
Chern connection, \cite[p.~11, Proposition 4.9]{Ko},
on the tautological line bundle on ${\mathbb P}(W)$
coincides with the Fubini--Study K\"ahler form on ${\mathbb P}(W)$. In other words, the curvature
form is positive. Let $\mu$ denote the curvature of the line bundle
$V\, \longrightarrow\, {\mathbb G}\, =\, {\mathbb P}(H^0(X,K_X))$ in \eqref{z3} equipped with the
Hermitian form $H_0$ constructed in \eqref{z4}. Since $\mu$ is positive, $\rho^*\mu$ is
nonnegative, and it is strictly positive wherever
the differential $d\rho$ is nonzero. Note that the curvature of the Chern connection on the
line bundle $\rho^*V$ equipped with the Hermitian structure $\rho^*H_0$ coincides with the
pulled back form $\rho^*\mu$.

If $X$ is not hyperelliptic, then $\rho$ is an embedding \cite[p.~11--12]{ACGH}, \cite[p.~247]{GH}.
So $\rho^*\mu$ is a positive form
on $X$. Since the curvature of the Chern connection on $(\rho^*V,\, \rho^*H_0)$ coincides with
$\rho^*\mu$, from Theorem \ref{thm1} it follows that the curvature form
$\Theta\cdot \varphi^*\omega_0$, for the K\"ahler structure $\varphi^*\omega_0$,
coincides with $-\rho^*\mu$. So $\Theta$ is strictly negative everywhere on $X$.

Now take $X$ to be hyperelliptic. Let $\iota\, :\, X\, \longrightarrow\, X$ be the hyperelliptic involution. The map $\rho$ factors as
$$ X\, \longrightarrow\, X/\iota \, \stackrel{\rho'}{\longrightarrow}\, {\mathbb P}(H^0(X,K_X))\, , $$
and $\rho'$ is an embedding \cite[p.~11]{ACGH}. In particular, the differential $d\rho$ vanishes exactly on the $2(g+1)$ points of $X$
fixed by the hyperelliptic involution $\iota$. Hence the above 
argument gives that the function $\Theta$ vanishes exactly on the $2(g+1)$ points of
$X$ fixed by $\iota$, and it strictly negative outside these $2(g+1)$ points.
\end{proof}

%%%%%%%%%%%%%%%%%%%%%%%%%%%%%%%%%%%%%%%%%%%%%%%%%%%%%%%%%%%%%%%%

\end{document}